\documentclass[12pt]{amsart}

\usepackage[utf8]{inputenc}
\usepackage[T1]{fontenc}
\usepackage[english]{babel}
\usepackage{amsfonts,amsmath,amssymb}
\usepackage{amsthm}
\usepackage[mathscr]{eucal}
\usepackage{tikz}
\usetikzlibrary{calc}
\usepackage[hang]{subfigure}
\usepackage{xspace}
\usepackage{booktabs,tabularx}
\usepackage{hyperref}

\newcommand\RR{{\mathbb R}}
\newcommand\ZZ{{\mathbb Z}}
\newcommand{\smallSetOf}[2]{\{#1:#2\}} 
\DeclareMathOperator\conv{conv}

\theoremstyle{plain}
\newtheorem{theorem}{Theorem}

\newtheorem{corollary}[theorem]{Corollary}
\newtheorem{lemma}[theorem]{Lemma}

\theoremstyle{definition}

\newtheorem{example}[theorem]{Example}
\newtheorem{remark}[theorem]{Remark}

\title[Foldable Triangulations of Lattice Polygons]{Foldable Triangulations of\\Lattice Polygons}

\author[Joswig and Ziegler]{Michael Joswig \and G\"unter M. Ziegler}
\address{Michael Joswig, Fachbereich Mathematik, TU Darmstadt, Dolivostr.~15, 64293 Darmstadt, Germany}
\email{joswig@mathematik.tu-darmstadt.de}
\address{G\"unter M. Ziegler, Institut f\"ur Mathematik, Freie Universit\"at Berlin, Arnimallee~2, 14195 Berlin, Germany}
\email{ziegler@math.fu-berlin.de}

\thanks{The first author is supported by the DFG Priority Program 1489 ``Experimental Methods in
  Algebra, Geometry and Number Theory''. The second author is supported by the European Research
  Council under the European Union's Seventh Framework Programme (FP7/2007-2013)/\allowbreak ERC
  Grant agreement no.~247029-SDModels and by the DFG Research Center \textsc{Matheon} ``Mathematics
  for Key Technologies'' in Berlin.}

\begin{document}

\begin{abstract}
  We give a simple formula for the signature of a foldable triangulation of a lattice
  polygon in terms of its boundary.  This yields lower bounds on the number of real roots of
  certain systems of polynomial equations known as ``Wronski systems''.
\end{abstract}

\maketitle

\section{Introduction}
\noindent
Let $P$ be a $d$-dimensional \emph{lattice polytope} in $\RR^d$; that is, $P$ is the convex hull of
finitely many points in $\ZZ^d$.  Further, let $\Delta$ be a triangulation of~$P$.  We call $\Delta$
\emph{dense} if the points $P\cap\ZZ^d$ are precisely the vertices of~$\Delta$.  Moreover, we call
$\Delta$ \emph{foldable} if its dual graph is bipartite; equivalently, the vertices of the
$1$-skeleton considered as an abstract graph is $(d{+}1)$-colorable \cite{Joswig02}.  In this case
we may distinguish between ``black'' and ``white'' $d$-dimensional cells of~$\Delta$.  The
\emph{signature} $\sigma(\Delta)$ is defined as the absolute value of the difference of the numbers
of black and of white $d$-dimensional cells of $\Delta$ that have odd normalized volume.  The
\emph{normalized volume} of $P$ equals $d!$ times the Euclidean volume of $P$.  As $P$ is a lattice
polytope the normalized volume is an integer. Soprunova and Sottile~\cite{SoprunovaSottile06} prove
that $\sigma(\Delta)$ is a lower bound for the number of real roots of so-called ``Wronski systems''
of $d$ polynomials over the reals in $d$ unknowns associated with $P$ and $\Delta$.  For $d=1$ that
result amounts to the basic fact that every univariate real polynomial of odd degree has at least
one real zero.  In general, the signature $\sigma(\Delta)$ can be interpreted as the topological
degree of the map which ``folds'' the foldable triangulation~$\Delta$ onto the standard simplex of
the same dimension; see \cite[\S7.3]{Sottile12} and the references given there.  In various cases
the Soprunova--Sottile bound is sharp for generic coefficients.  Here we are concerned with the
planar case $d=2$, where we provide a new lower bound for the signature, and then an explicit
example of a Wronski system for which the bound is sharp.

In the literature foldable abstract simplicial complexes are also called ``balanced''; for instance,
see \cite{Stanley}.

\section{Lattice Polygons}
\noindent
Let $p$ and $q$ be lattice points in $\ZZ^2$.  We say that the line segment $[p,q]=\conv\{p,q\}$ is 
\emph{of~type~$X$} if the first coordinate of $p-q$ is odd and the other one is even.  Similarly, the  
segment is \emph{of~type~$Y$} if the second coordinate is odd and the other one is even.  For a  
segment \emph{of~type~$XY$} both coordinates are odd.

The key to our main result is the following observation.

\begin{lemma}\label{lemma1}
  Let $T$ be a lattice triangle of odd normalized area in the plane.  Then $T$ has precisely one
  edge of type~$X$, one of type~$Y$ and one of type~$XY$.
\end{lemma}

\begin{proof}
  Up to a translation we can assume that the vertices of $T$ are $(0,0)$, $(a,b)$ and $(c,d)$ with
  $a,b,c,d\in\ZZ$.  Then the normalized area of~$T$ is given by the absolute value of the determinant
  \[
  \det\begin{pmatrix} a & b \\ c & d \end{pmatrix} \ = \ ad-bc \, ,
  \]
  and our precondition says that this is odd.  It follows that the parities of the two products $ad$
  and $bc$ are distinct.  After possibly exchanging the vertices
  $(a,b)$ and $(c,d)$ we may assume that $ad$ is odd, and thus both $a$ and $d$ are odd.  

  If both $b$ and $c$ are even then the edge $[(0,0),(a,b)]$ is of
  type~$X$, the edge $[(0,0),(c,d)]$ is of type~$Y$, and the third edge $[(a,b),(c,d)]$ is of
  type~$XY$.  If, however, $b$ is odd and $c$ is even, then $[(0,0),(a,b)]$ is of type~$XY$,
  $[(0,0),(c,d)]$ is of type~$Y$ and $[(a,b),(c,d)]$ is of type~$X$.  The situation is similar if
  $b$ is even and $c$ is odd.
\end{proof}

\begin{remark}
  The preceding lemma admits the following generalization.  Let $T$ be an arbitrary lattice
  triangle.  Then the normalized area of $T$ is congruent modulo $2$ to the number of edges of type
  $\tau$, where $\tau\in\{X,Y,XY\}$ is any type.  We omit the proof, which is similar to the above.
\end{remark}

Now let $d=2$ and let $P$ be a lattice polygon with a dense and foldable triangulation~$\Delta$. 
For $d\le2$ foldable triangulations can also be characterized by the fact that  all interior vertices of 
its graph have even degree.
(Indeed, the interesting case is $d=2$, as for $d=1$ all triangulations are foldable.)
It~is also special to dimensions $d\le2$ that a triangulation is dense if and only if it is unimodular,
that is, each triangle has normalized area one. 
As~$\Delta$ is foldable we may distinguish between
``black'' and ``white'' triangles.  We can also assign colors, either black or white, to the
boundary edges of a foldable triangulation, according to the color of the unique triangle they are
contained in.  Thus we are ready for our main result.

\begin{figure}[th] 
\begin{tikzpicture}
\usetikzlibrary{calc}
 
\coordinate (v1) at (0,1);
\coordinate (v2) at (0,2); 
\coordinate (v4) at (1,1);
\coordinate (v5) at (1,2);
\coordinate (v6) at (1,3); 
\coordinate (v8) at (2,1);
\coordinate (v9) at (2,2);
\coordinate (v10) at (2,3);
\coordinate (v11) at (2,4);
\coordinate (v12) at (3,0);
\coordinate (v13) at (3,1);
\coordinate (v14) at (3,2);
\coordinate (v15) at (3,3);
\coordinate (v16) at (3,4);
\coordinate (v17) at (3,5);
\coordinate (v18) at (4,0);
\coordinate (v19) at (4,1);
\coordinate (v20) at (4,2);
\coordinate (v21) at (4,3);
\coordinate (v22) at (4,4);
\coordinate (v23) at (5,0);
\coordinate (v24) at (5,1);
\coordinate (v25) at (5,2);
\coordinate (v26) at (5,3);
\coordinate (v27) at (5,4);

\tikzstyle{edge} = [draw,thick,-,black]

\filldraw[edge,fill=gray] (v1) -- (v4) -- (v12) -- cycle;
\filldraw[edge,fill=white] (v8) -- (v4) -- (v12) -- cycle; 
\filldraw[edge,fill=white] (v1) -- (v4) -- (v5) -- cycle;
\filldraw[edge,fill=white] (v2) -- (v5) -- (v6) -- cycle;
\filldraw[edge,fill=gray] (v1) -- (v2) -- (v5) -- cycle; 
\filldraw[edge,fill=gray] (v4) -- (v5) -- (v8) -- cycle;
\filldraw[edge,fill=white] (v5) -- (v8) -- (v9) -- cycle;

\filldraw[edge,fill=gray] (v5) -- (v6) -- (v11) -- cycle;
\filldraw[edge,fill=white] (v5) -- (v10) -- (v11) -- cycle;
\filldraw[edge,fill=gray] (v5) -- (v9) -- (v10) -- cycle;
\filldraw[edge,fill=white] (v9) -- (v10) -- (v14) -- cycle;
\filldraw[edge,fill=gray] (v10) -- (v11) -- (v15) -- cycle;
\filldraw[edge,fill=white] (v11) -- (v15) -- (v16) -- cycle;
 
\filldraw[edge,fill=gray] (v8) -- (v12) -- (v18) -- cycle;

\filldraw[edge,fill=gray] (v8) -- (v9) -- (v14) -- cycle;
\filldraw[edge,fill=white] (v8) -- (v13) -- (v14) -- cycle;
\filldraw[edge,fill=white] (v8) -- (v18) -- (v23) -- cycle;
\filldraw[edge,fill=gray] (v8) -- (v13) -- (v23) -- cycle;

\filldraw[edge,fill=gray] (v10) -- (v14) -- (v20) -- cycle;
\filldraw[edge,fill=white] (v10) -- (v15) -- (v20) -- cycle;

\filldraw[edge,fill=white] (v16) -- (v20) -- (v21) -- cycle;
\filldraw[edge,fill=gray] (v15) -- (v16) -- (v20) -- cycle;

\filldraw[edge,fill=gray] (v13) -- (v14) -- (v19) -- cycle;
\filldraw[edge,fill=white] (v14) -- (v19) -- (v20) -- cycle;

\filldraw[edge,fill=white] (v13) -- (v19) -- (v23) -- cycle;

\filldraw[edge,fill=gray] (v19) -- (v20) -- (v23) -- cycle;
\filldraw[edge,fill=white] (v20) -- (v23) -- (v24) -- cycle;

\filldraw[edge,fill=gray] (v20) -- (v21) -- (v24) -- cycle;
\filldraw[edge,fill=white] (v21) -- (v24) -- (v25) -- cycle;

\filldraw[edge,fill=gray] (v16) -- (v21) -- (v22) -- cycle;
\filldraw[edge,fill=white] (v21) -- (v22) -- (v26) -- cycle;

\filldraw[edge,fill=gray] (v21) -- (v25) -- (v26) -- cycle;
\filldraw[edge,fill=gray] (v22) -- (v26) -- (v27) -- cycle;

\foreach \point in {v1,v2,v4,v5,v6,v8,v9,v10,v11,v12,v13,v14,v15,v16,v18,v19,v20,v21,v22,v23,v24,v25,v26,v27}
    \fill[black] (\point) circle (2pt);

\node [inner sep=1pt,label=below:$XY$] (X0) at ($ (v1)!.5!(v12) $) {};
	 
\node [inner sep=1pt,label=below:$X$] (X3) at ($ (v12)!.5!(v18) $) {};
\node [inner sep=1pt,label=below:$X$] (X4) at ($ (v18)!.5!(v23) $) {};
\node [inner sep=1pt,label=above:$X$] (X5) at ($ (v11)!.5!(v16) $) {};
\node [inner sep=1pt,label=above:$X$] (X6) at ($ (v16)!.5!(v22) $) {};
\node [inner sep=1pt,label=above:$X$] (X7) at ($ (v22)!.5!(v27) $) {};
 
\node [inner sep=1pt,label=left:$Y$] (Y1) at ($ (v1)!.5!(v2) $) {};
\node [inner sep=1pt,label=right:$Y$] (Y2) at ($ (v23)!.5!(v24) $) {};
\node [inner sep=1pt,label=right:$Y$] (Y3) at ($ (v24)!.5!(v25) $) {};
\node [inner sep=1pt,label=right:$Y$] (Y4) at ($ (v25)!.5!(v26) $) {};
\node [inner sep=1pt,label=right:$Y$] (Y5) at ($ (v26)!.5!(v27) $) {};

\node [inner sep=0pt,label=above left:$XY$] (XY0) at ($ (v2)!.5!(v6) $) {};
\node [inner sep=0pt,label=above left:$XY$] (XY1) at ($ (v6)!.5!(v11) $) {}; 

\end{tikzpicture}
  \caption{A lattice polygon with a dense and foldable triangulation}
  \label{fig:polygon}
\end{figure}
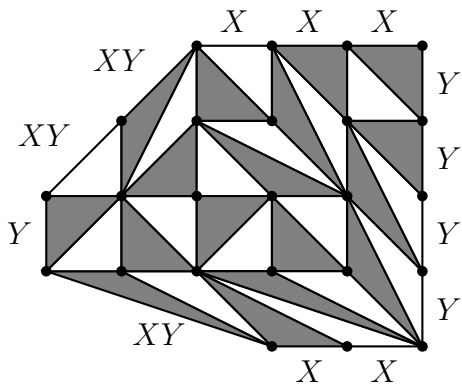

\begin{theorem}
  Let $\Delta$ be a dense and foldable triangulation of a planar lattice polygon $P$.  Then the
  signature $\sigma(\Delta)$ equals the absolute value of the difference between the numbers of black
  and of white edges of type~$\tau$, for any fixed $\tau\in\{X,Y,XY\}$.
\end{theorem}

\begin{proof}
  By Lemma~\ref{lemma1} each triangle of $\Delta$ has precisely one edge of type~$X$, one edge of type~$Y$,
  and one edge of type~$XY$. Fix $\tau$. Each $\tau$ edge is either an interior edge, in which case
  it is contained in exactly one black and one white triangle, or it is on the boundary.
  Conversely, each triangle contains a $\tau$ edge, and so the interior $\tau$ edges induce a
  perfect matching on the interior triangles of $\Delta$.  Their total contribution to the signature
  is zero.  The claim follows.
\end{proof}

This result can be interpreted as a combinatorial version of Green's Theorem on integrating a
potential over a simply connected planar domain by integrating over the boundary.  It also bears
some similarity with Pick's Theorem, which states that a lattice polygon of Euclidean area $A$ and
$I$ interior lattice points has exactly
\begin{equation}\label{eq:pick}
  B \ = \ 2\cdot (A-I+1)
\end{equation}
lattice points on the boundary; see Beck and Robbins \cite[\S2.6]{BeckRobbins07} as well as the 
``Green's Theorem style'' proof of Pick's Theorem by Blatter~\cite{Blatter}.

\begin{example}
  Figure~\ref{fig:polygon} shows a lattice hexagon with a dense and foldable triangulation.  It has
  17 black triangles and 16 white ones; so the signature equals one.  At the same time it has three
  black and two white boundary edges of type $X$, three black and two white boundary edges of type
  $Y$, as well as two black and one white boundary edges of type $XY$.  For each type there is a
  surplus of exactly one black edge.
\end{example}

The following corollary solves a problem which the first author presented at the Oberwolfach
meeting on Geometric and Topological Combinatorics in 2007.  The authors gratefully acknowledge the
repeated hospitality of the Mathematisches For\-schungs\-institut; this paper was written during another
stay at the Institut in 2012.

\begin{corollary}
  Let $P$ be an axis-parallel lattice rectangle in the plane.  Then the signature of any dense and
  foldable triangulation of $P$ vanishes.
\end{corollary}

\begin{proof}
  There are no $XY$ edges in the boundary.
\end{proof}

Our main result entails a general upper bound on the signature for lattice polygons.

\begin{corollary}
  Let $\Delta$ be a dense and foldable triangulation of a planar lattice polygon $P$.  Then the
  signature $\sigma(\Delta)$ is bounded from above by
  \[
  \lfloor \tfrac{2}{3} (A-I+1) \rfloor
  \]
  where $A$ is the Euclidean area and $I$ is the number of interior lattice points of $P$.
\end{corollary}

\begin{proof}
  By Pick's Theorem~\eqref{eq:pick} the number $B$ of lattice points on the boundary of $P$ equals
  $A-I+1$.  Now $B$ is also the number of boundary edges of the triangulation $\Delta$, and so there
  is some type $\tau\in\{X,Y,XY\}$ such that the number of $\tau$-edges in the boundary does not
  exceed $\lfloor \frac{1}{3} B \rfloor$.  Since the signature of $\Delta$ cannot exceed the number of
  boundary edges of type~$\tau$ the claim follows.
\end{proof}

\begin{example}  
  There are lattice polygons for which this bound is sharp.  Consider the triangle
  $P=\conv\{(0,0),(1,0),(0,1)\}$.  Here $A=\frac{1}{2}$, $I=0$, and
  $\frac{2}{3}(A-I+1)=\frac{2}{3}(\frac{1}{2}-0+1)=1$, which equals the signature of the trivial
  triangulation of $P$.  More generally, the dilates $n P$ for arbitrary integral $n\ge 1$ are
  lattice triangles with a triangulation $\Delta_n$ induced by the integral translates of the
  coordinate axes and the diagonal line $\smallSetOf{(x,y)\in\RR^2}{x+y=0}$.  This triangulation is
  dense and foldable.  All the boundary edges share the same color, and we have exactly $n$ boundary
  edges of each type.  The signature of $\Delta_n$ equals $n$ and this coincides with our bound.
  However, the example of axis-parallel rectangles shows that our bound on the signature can be
  arbitrarily bad in general.  The triangle $\Delta_2$ is shown in Figure~\ref{fig:triangle} below.
\end{example}
 
% We leave it to the readers to explore the implications of our result to bivariate polynomial
% systems and the associated toric varieties.

\section{Wronski Polynomials}
\noindent
In the remainder of this note we will sketch how our results are related to bivariate polynomial
systems.  Omitting most of the technical details we aim at explaining one simple example.

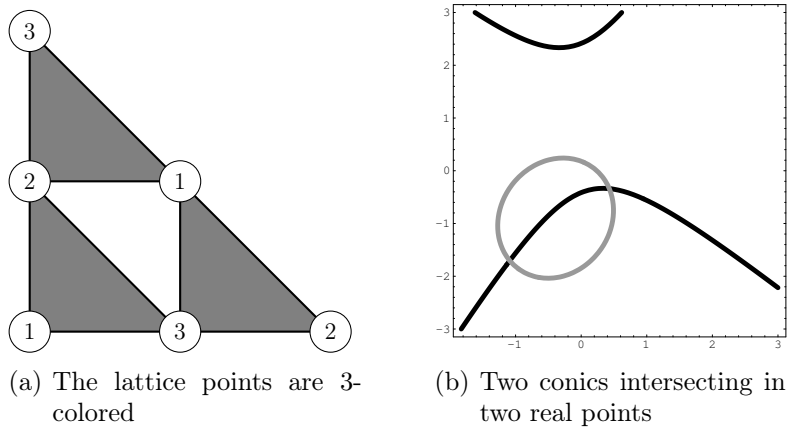
\begin{figure}[th]
  \subfigure[The lattice points are $3$-colored]{%
    \begin{tikzpicture}[scale=2]
      \usetikzlibrary{calc}
      
      \tikzstyle{edge} = [draw,thick,-,black]
      \tikzstyle{node} = [circle,draw,fill=white,scale=0.75]
      
      \newcommand\size{2}
      
      \foreach \i in {1,...,\size} {
        \foreach \j in {\i,...,\size} {
          \filldraw[edge,fill=gray] (\i - 1, \size - \j) -- (\i, \size - \j) -- (\i - 1, \size - \j + 1) -- cycle;
        }
      }
      
      \node[node] at (0,0) {1};
      \node[node] at (0,1) {2};
      \node[node] at (0,2) {3};
      \node[node] at (1,0) {3};
      \node[node] at (1,1) {1};
      \node[node] at (2,0) {2};
    \end{tikzpicture}
    \label{fig:triangle}
  }
  \qquad
  \subfigure[Two conics intersecting in two real points]{\includegraphics[width=.35\textwidth]{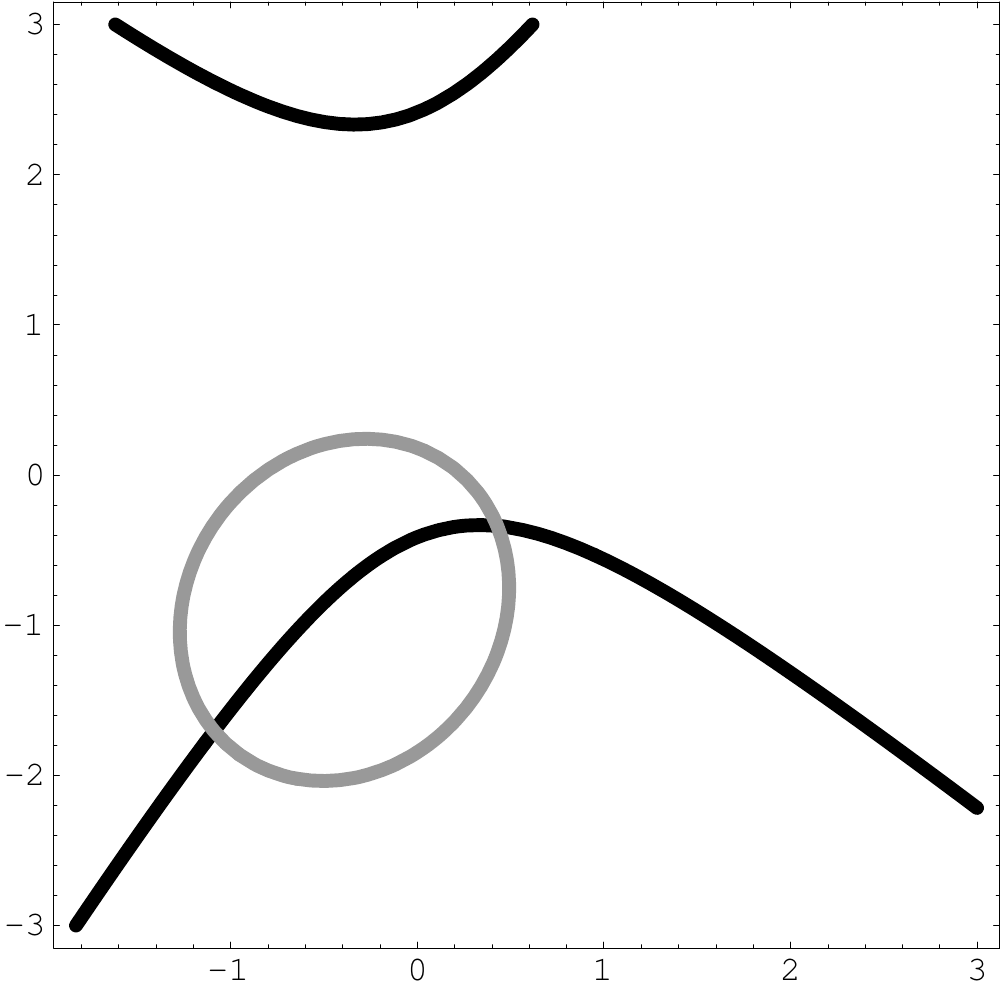}\label{fig:curves}}
  \caption{Lattice triangle and curves of a related bivariate Wronski system}
  \label{fig:wronski}
\end{figure}

Let $P$ be a lattice polygon with a dense and foldable triangulation $\Delta$.  As $\Delta$ is
foldable, the lattice points in $P$ can be assigned numbers from the set $\{1,2,3\}$ such that two
points receive distinct labels whenever they lie on a common edge of $\Delta$.  That is to say, the
primal graph of $\Delta$ (which is clearly planar) is $3$-colorable.  A proof follows from Heawood's
criterion on the $3$-colorability of maximal planar graphs since $\Delta$ being foldable forces that
the dual graph of $\Delta$ does not contain any odd cycles \cite{Heawood}.  In fact, up to
relabeling such an assignment, $c:P\cap\ZZ^2\to\{1,2,3\}$, is unique.  Now for any choice of
non-zero real numbers $\gamma_1,\gamma_2,\gamma_3$ the polynomial
\[
\sum_{(i,j)\in P\cap\ZZ^2}\gamma_{c(i,j)}x^iy^j
\]
is a bivariate \emph{Wronski polynomial} with respect to $\Delta$.  Since the coefficients
$\gamma_i$ are non-zero the Newton polygon of such a polynomial is $P$.  A bivariate 
\emph{Wronski system} consists of two bivariate Wronski polynomials with respect to the same triangulation
$\Delta$.  Provided that the coefficients are chosen sufficiently generic the two polynomials have
just finitely many complex roots in common.  A classical result of Kushnirenko says that in this
case the number of common complex roots coincides with the normalized area of the polygon $P$; see
\cite{BKK} and \cite[Thm.~3.2]{Sottile12}.  Under additional assumptions Soprunova and Sottile show
that the signature $\sigma(\Delta)$ yields a lower bound for the number of common \emph{real} roots;
see \cite{SoprunovaSottile06} and \cite[Thm.~7.13]{Sottile12}.

\begin{example}
  Let $P$ be the lattice triangle $\conv\{(0,0),(2,0),(0,2)\}$ with the dense and foldable
  triangulation $\Delta$ shown in Figure~\ref{fig:triangle}.  This figure also defines the labeling
  $c$ of the lattice points.  For instance, $c(0,0)=1$ and $c(1,0)=3$.  Choosing $(1,-1,2)$ and $(1,3,5)$ 
  as coefficient vectors we obtain the Wronski system
  \begin{equation}\label{eq:wronski}
    \begin{aligned}
      (1+xy)- (x+y^2)+2(x^2+y) \ &= \ 0\\
      (1+xy)+3(x+y^2)+5(x^2+y) \ &= \ 0
    \end{aligned}
  \end{equation}
  of bivariate polynomial equations.  The real conic defined by the first polynomial is a hyperbola,
  the second one defines an ellipse; see Figure~\ref{fig:curves}.  The normalized area of $P$ is
  four and so, by Kushnirenko's Theorem, the polynomial system \eqref{eq:wronski} has exactly four
  complex solutions (the coefficients \emph{are} sufficiently generic).  The signature of $\Delta$
  equals two and thus, by Soprunova and Sottile, at least two of the complex solutions must be real
  (again all extra conditions are met in this case).
\end{example}

The results on Wronski systems generalize to arbitrary dimension $d$, in which case we have $d$
polynomials in $d$ indeterminates.  It is instructive to look at the case $d=1$.  If $f$ is a
univariate real polynomial with non-vanishing constant coefficient its $1$-dimensional Newton
polytope is the interval from $0$ to the degree of $f$.  Its normalized volume is the length of the
interval, which equals the degree.  If the coefficients of $f$ are sufficiently generic there are
degree many complex roots.  In this sense Kushnirenko's Theorem generalizes the Fundamental Theorem
of Algebra.  To obtain a lower bound for the number of \emph{real} roots of $f$, consider the unique
dense triangulation of the interval $[0,\deg f]$ into unit intervals.  This is clearly foldable, and
its signature is the parity of the degree.  This means that the result by Soprunova and Sottile
generalizes the well known fact that a univariate real polynomial of odd degree has at least one
real root.


\begin{thebibliography}{1}

\bibitem{BeckRobbins07}
Matthias Beck and Sinai Robins, \emph{Computing the Continuous Discretely},
  Undergraduate Texts in Mathematics, Springer, New York, 2007.

\bibitem{BKK} David Bernstein, Anatolii G. Kushnirenko, and Askold Khovanskii, 
\emph{Newton polytopes}, Usp.\ Math.\ Nauk.\ \textbf{31} (1976), 201--202.

\bibitem{Blatter}
Christian Blatter, \emph{Another proof of Pick’s area theorem},
Math.\ Magazine \textbf{70} (1997), 200. 

\bibitem{Heawood}
Percy J. Heawood, \emph{On the four-colour map theorem}, Quarterly J.\ Math.\ \textbf{29} (1898), 270--285.

\bibitem{Joswig02}
Michael Joswig, \emph{Projectivities in simplicial complexes and colorings of
  simple polytopes}, Math.\ Zeitschrift \textbf{240} (2002), 243--259.

% \bibitem{Monsky70} Paul Monsky, \emph{On dividing a square into triangles}, Amer. Math. Monthly
%   \textbf{77} (1970), 161--164.

\bibitem{Stanley} Richard P. Stanley, \emph{Combinatorics and Commutative Algebra}, 2nd edition,
  Progress in Mathematics \textbf{41}, Birkh\"auser Boston, Inc., Boston, MA, 1996.

\bibitem{SoprunovaSottile06}
Evgenia Soprunova and Frank Sottile, \emph{Lower bounds for real solutions to
  sparse polynomial systems}, Advances in Math.\ \textbf{204} (2006), 116--151.

\bibitem{Sottile12}
Frank Sottile, \emph{Real Solutions to Equations from Geometry},
   University Lecture Series, American Math.\ Soc.,
   Providence, RI, 2011.

\end{thebibliography}
\end{document}